\documentclass[12pt,reqno]{article}

\usepackage[usenames]{color}
\usepackage{amssymb}
\usepackage{amsmath}
\usepackage{amsthm}
\usepackage{amsfonts}
\usepackage{amscd}
\usepackage{graphicx}

\usepackage[colorlinks=true,
linkcolor=webgreen,
filecolor=webbrown,
citecolor=webgreen]{hyperref}

\definecolor{webgreen}{rgb}{0,.5,0}
\definecolor{webbrown}{rgb}{.6,0,0}

\usepackage{color}
\usepackage{fullpage}
\usepackage{float}
\usepackage{graphics}
\usepackage{latexsym}
\usepackage{breakurl}

\begin{document}

\theoremstyle{plain}
\newtheorem{theorem}{Theorem}
\newtheorem{corollary}[theorem]{Corollary}
\newtheorem{lemma}{Lemma}
\newtheorem{example}{Example}
\newtheorem*{remark}{Remark}

\begin{center} 
\vskip 1cm
{\LARGE\bf Fibonacci sums modulo 5 \\ }

\vskip 1cm

{\large
Kunle Adegoke \\
Department of Physics and Engineering Physics \\ Obafemi Awolowo University, 220005 Ile-Ife, Nigeria \\
\href{mailto:adegoke00@gmail.com}{\tt adegoke00@gmail.com}

\vskip 0.2 in

Robert Frontczak \\
Independent Researcher \\ Reutlingen,  Germany \\
\href{mailto:robert.frontczak@web.de}{\tt robert.frontczak@web.de}

\vskip 0.2 in

Taras Goy  \\
Faculty of Mathematics and Computer Science \\ Vasyl Stefanyk Precarpathian National University, Ivano-Frankivsk, Ukrai\-ne \\
\href{mailto:taras.goy@pnu.edu.ua}{\tt taras.goy@pnu.edu.ua}
}

\end{center}

\vskip .2 in

\begin{abstract}
We develop closed form expressions for various finite binomial Fibonacci and Lucas sums depending 
on the modulo 5 nature of the upper summation limit. Our expressions are inferred from some trigonometric identities.
\end{abstract}

\vspace{2pt}


\bigskip

\section{Preliminaries}

As usual, the Fibonacci numbers $F_n$ and the Lucas numbers $L_n$ are defined, for $n\in\mathbb Z$, 
by the following recurrence relations for $n\ge 2$: 
\begin{align*}
&F_n = F_{n-1}+F_{n-2},\quad F_0=0,\,\,\, F_1=1,\\
&L_n = L_{n-1}+L_{n-2},\quad L_0=2,\,\,\, L_1=1.
\end{align*}
For negative subscripts we have $
F_{-n}=(-1)^{n-1}F_n$ and $L_{-n}=(-1)^n L_n$.

Throughout this paper, we denote the golden ratio by $\alpha=\frac{1+\sqrt 5}2$ and write $\beta=\frac{1-\sqrt 5}2=-\frac{1}{\alpha}$. 
The Fibonacci and Lucas numbers possess the explicit formulas (Binet forms)
\begin{equation*}
F_n = \frac{\alpha^n - \beta^n }{\alpha - \beta },\quad L_n = \alpha^n + \beta^n,\quad n\in\mathbb Z.
\end{equation*} 

The sequences $\{F_n\}_{n\geq 0}$ and $\{L_n\}_{n\geq 0}$ are indexed in the On-Line Encyclopedia of Integer Sequences 
\cite{OEIS} as entries {A000045} and A000032, respectively. 
For more information we refer to Koshy \cite{Koshy} and Vajda \cite{Vajda} who have written excellent books dealing with Fibonacci and Lucas numbers. 

There exists a countless number of binomial sums involving Fibonacci and Lucas numbers.
For some new articles in this field we refer to the papers \cite{Adegoke1,Adegoke2,Adegoke3,Bai}.

In this paper, we introduce closed form expressions for finite Fibonacci and Lucas sums involving different kinds of binomial coefficients and depending on the the modulo 5 nature of the upper summation limit. Our expressions are derived from various trigonometric identities, particularly utilizing Waring formulas and Chebyshev polynomials of the first and second kinds. We also present some series involving Bernoulli polynomials. 

We note that some of our results were  announced without proofs in \cite{Kravchuk_Conf}.

\section{Fibonacci sums modulo 5 from the $\sin nx$ and $\cos nx$ expansions}

We begin with a known lemma \cite[1.331(3) and 1.331(1)]{grad07}.
\begin{lemma} If $n$ is a positive integer, then
\begin{align}
&\sum_{k = 1}^{\left\lfloor {n/2} \right\rfloor } {\frac{{( - 1)^{k - 1} n}}{k}\binom{n - k - 1}{k - 1}2^{n - 2k - 1} \cos ^{n - 2k} x}  = 2^{n - 1} \cos ^n x - \cos nx, \label{eq.zbuts4b}\\
&\sum_{k = 0}^{\left\lfloor {(n - 1)/2} \right\rfloor } {( - 1)^k \binom{n - k - 1}k2^{n - 2k - 1} \cos ^{n - 2k - 1} x}  = \frac{{\sin  nx}}{\sin x}\label{eq.dbuskhn}.
\end{align}
\end{lemma}
\begin{lemma}
If $n$ is an integer, then
\begin{align}
&\cos{\Big(\frac{n\pi}{5}\Big)} =  \begin{cases}
  (-1)^n,&\text{\rm if $n\equiv 0\pmod 5$};  \\ 
  (-1)^{n - 1}\alpha/2,&\text{\rm if $n\equiv 1$ or $4\pmod 5$};  \\ 
  (-1)^{n - 1}\beta/2,&\text{\rm if $n\equiv 2$ or $3\pmod 5$}; \label{eq.kxkaziy}  
 \end{cases}\\
 &\cos{\Big(\frac{2n\pi}{5}\Big)} =  \begin{cases}
  1,&\text{\rm if $n\equiv 0\pmod 5$};  \\ 
  -\beta/2,&\text{\rm if $n\equiv 1$ or $4\pmod 5$};  \\ 
  -\alpha/2,&\text{\rm if $n\equiv 2$ or $3\pmod 5$}.\label{eq.argxf9e} 
 \end{cases}
\end{align}
\end{lemma}
\begin{proof}
Relations stated in \eqref{eq.kxkaziy} can be proved easily by elementary methods. 
For instance, they follow by applying the addition theorem for the cosine function
\begin{equation*}
\cos (a+b) = \cos a\cos b - \sin a\sin b
\end{equation*}
combined with the special values
\begin{equation*}
\cos \Big (\frac{\pi}{5} \Big ) = \frac{\alpha}{2},\quad \cos \Big (\frac{2\pi}{5} \Big ) = - \frac{\beta}{2}, \quad 
\cos \Big (\frac{3 \pi}{5} \Big ) = \frac{\beta}{2}, \quad \cos \Big (\frac{4 \pi}{5} \Big ) = -\frac{\alpha}{2}.
\end{equation*}

Relations stated in \eqref{eq.argxf9e} follow directly from \eqref{eq.kxkaziy}. 
\end{proof}

In our first main results we state Lucas (Fibonacci) identities involving binomial coefficient and  additional parameter.
\begin{theorem}\label{th1}
If $n$ is a positive integer and $t$ is any integer, then
\begin{align*}
&n\sum_{k = 1}^{\left\lfloor {n/2} \right\rfloor } \frac{{( - 1)^{k - 1}}}{k}\binom{n - k - 1}{k - 1}L_{n - 2k + t} 
  =  \begin{cases}
  L_{n + t}  - ( - 1)^{n} 2L_t,&\text{\rm if $n\equiv 0\pmod 5$};  \\ 
  L_{n + t}  + ( - 1)^n L_{t + 1},&\text{\rm if $n\equiv 1$ or $4\pmod 5$};  
  \\ 
  L_{n + t}  - ( - 1)^{n} L_{t - 1},&\text{\rm if $n\equiv 2$ or $3\pmod 5$}; 
  \end{cases} \\
&n\sum_{k = 1}^{\left\lfloor {n/2} \right\rfloor } {\frac{{( - 1)^{k - 1}}}{k}\binom{n - k - 1}{k - 1}F_{n - 2k + t} }
  =  \begin{cases}
   F_{n + t}  - ( - 1)^{n} 2F_t,&\text{\rm if $n\equiv 0\pmod 5$};  \\ 
   F_{n + t}  + ( - 1)^n F_{t + 1},&\text{\rm if $n\equiv 1$ or $4\pmod 5$};  \\ 
    F_{n + t}  - ( - 1)^{n} F_{t - 1},&\text{\rm if $n\equiv 2$ or $3\pmod 5$}.
\end{cases}
\end{align*}
\end{theorem}
\begin{proof}
Set $x=\pi/5$ in \eqref{eq.zbuts4b} and use \eqref{eq.kxkaziy} and the fact that
\begin{equation}\label{eq.fieo8cp}
2\alpha^r =L_r + F_r\sqrt 5,\quad 2\beta^r =L_r - F_r\sqrt 5
\end{equation}
for any integer $r$.
\end{proof}

We proceed with some corollaries.
\begin{corollary}  
If $n$ is a positive integer, then
\begin{align*}
&n\sum_{k = 1}^{\left\lfloor {n/2} \right\rfloor } {\frac{{( - 1)^{k - 1}}}{k}\binom{n - k - 1}{k - 1}F_{2k} }
  =  \begin{cases}
   -2F_n,&\text{\rm if $n\equiv 0\pmod 5$};  \\ 
   -F_{n - 1},&\text{\rm if $n\equiv 1$ 
   	or $4\pmod 5$};
   	\\ 
   F_{n + 1},&\text{\rm if $n\equiv 2$ or $3\pmod 5$};  
   \end{cases}\\
&n\sum_{k = 1}^{\left\lfloor {n/2} \right\rfloor } {\frac{{( - 1)^{k - 1}}}{k}\binom{n - k - 1}{k - 1}F_{n - 2k + \delta} }=F_{n+\delta},
\end{align*}
where
\begin{equation*}
\delta=  \begin{cases}
   0,&\text{\rm if $n\equiv 0\pmod 5$};  \\ 
   -1,&\text{\rm if $n\equiv 1$ or $4\pmod 5$};\\  
    1,&\text{\rm if $n\equiv 2$ or $3\pmod 5$}. 
   \end{cases}
\end{equation*}
\end{corollary}
\begin{corollary}
If $n$ is a positive integer, then
\begin{align*}
&n\sum_{k = 1}^{\left\lfloor {n/2} \right\rfloor} \frac{{( - 1)^{k - 1} }}{k}\binom{n - k - 1}{k - 1}L_{n - 2k - 1} 
  =  \begin{cases}
  L_{n - 1}  - ( - 1)^{n} 3,&\text{\rm if $n\equiv 2$ or $3\pmod 5$};\\
  L_{n - 1}  + ( - 1)^n 2,&\text{\rm otherwise};
  \end{cases}\\ 
&n\sum_{k = 1}^{\left\lfloor {n/2} \right\rfloor } \frac{{( - 1)^{k - 1} }}{k}\binom{n - k - 1}{k - 1}L_{n - 2k + 1} 
  =  \begin{cases}
  L_{n + 1}  + ( - 1)^n 3,&\text{\rm if $n\equiv 1$ or $4\pmod 5$};\\ 
 	L_{n + 1} - ( - 1)^{n} 2,&\text{\rm otherwise}; 
 	 \end{cases}\\ 
&n\sum_{k = 1}^{\left\lfloor {n/2} \right\rfloor } {\frac{{( - 1)^{k - 1}}}{k}\binom{n - k - 1}{k - 1}L_{n - 2k} }
  =  \begin{cases}
  L_n  - ( - 1)^{n} 4,&\text{\rm if $n\equiv 0\pmod 5$};  \\ 
  L_n  + ( - 1)^n,&\text{\rm otherwise}.
 	 \end{cases}
\end{align*}
\end{corollary}
\begin{lemma}\label{lem3}
If $n$ is an integer, then
\begin{align}
&\frac{\sin{(n\pi/5)}}{\sin{(\pi/5)}} =  \begin{cases}
  0,&\text{\rm if $n\equiv 0\pmod 5$};  \\ 
  (-1)^{\left\lfloor n/5\right\rfloor},&\text{\rm if $n\equiv 1$ or $4\pmod 5$};  \\ 
  (-1)^{\left\lfloor n/5\right\rfloor}\alpha,&\text{\rm if $n\equiv 2$ or $3\pmod 5$}; \label{eq.efvrm4x} 
 \end{cases}\\
 &\frac{\sin(3n\pi/5)}{\sin(3\pi/5)} =  \begin{cases}
  0,&\text{\rm if $n\equiv 0\pmod 5$};  \\ 
  (-1)^{\left\lfloor n/5\right\rfloor},&\text{\rm  if $n\equiv 1$ or $4\pmod 5$};  \\ 
  (-1)^{\left\lfloor n/5\right\rfloor}\beta,&\text{\rm if $n\equiv 2$ or $3\pmod 5$}. \label{eq.efvrm4y}
 \end{cases}
\end{align}
\end{lemma}

From Lemma \ref{lem3} we can deduce the following Lucas and Fibonacci binomial identities modulo 5.
\begin{theorem}\label{thmxxxyyy1}
If $n$ is a positive
integer and $t$ is any integer, then
\begin{align}
&\sum_{k = 0}^{\left\lfloor {n/2} \right\rfloor } {( - 1)^k \binom{n - k}{k} L_{n - 2k + t} }  =  \begin{cases}
( - 1)^{\left\lfloor {(n + 1)/5} \right\rfloor } L_t,&\text{\rm  if $n\equiv 0$ or $3\pmod 5$};  \label{Waring1} \\ 
 ( - 1)^{\left\lfloor {(n + 1)/5} \right\rfloor } L_{t + 1},&\text{\rm  if $n\equiv 1$ or $2\pmod 5$};\\  
  0, &\text{\rm  if $n\equiv 4\pmod 5$};
  \end{cases}\\
 &\sum_{k = 0}^{\left\lfloor {n/2} \right\rfloor } {( - 1)^k \binom{n - k}{k} F_{n - 2k + t} }  =  \begin{cases}
 ( - 1)^{\left\lfloor {(n + 1)/5} \right\rfloor } F_t,&\text{\rm  if $n\equiv 0$ or $3\pmod 5$}; \label{Waring2} \\ 
 ( - 1)^{\left\lfloor {n/5} \right\rfloor } F_{t + 1},&\text{\rm  if $n\equiv 1$ or $2\pmod 5$};\\
 0, &\text{\rm if $n\equiv 4\pmod 5$}.  
 \end{cases}
\end{align}
\end{theorem}
\begin{proof}
Set $x=\pi/5$ in \eqref{eq.dbuskhn},  use \eqref{eq.efvrm4x}, \eqref{eq.fieo8cp} and simplify.
\end{proof}

A variant of the Lucas and Fibonacci sums with even subscripts is stated as the next corollary.
\begin{corollary}
If $n$ is a positive integer, then
\begin{align*}
&\sum_{k = 0}^{\left\lfloor {n/2} \right\rfloor } {( - 1)^{n-k} \binom{n - k}kL_{2k} }  =  \begin{cases}
( - 1)^{\left\lfloor {(n + 1)/5} \right\rfloor} L_n,&\text{\rm if $n\equiv 0$ or $3\pmod 5$};  \\ 
 ( - 1)^{\left\lfloor {(n + 1)/5} \right\rfloor + 1} L_{n - 1},&\text{\rm if $n\equiv 1$ or $2\pmod 5$};\\
  0, &\text{\rm if $n\equiv 4\pmod 5$}; 
  \end{cases}\\
 &\sum_{k = 0}^{\left\lfloor {n/2} \right\rfloor } {( - 1)^{n-k} \binom{n - k}kF_{2k} }  =  \begin{cases}
 ( - 1)^{\left\lfloor {(n + 1)/5} \right\rfloor} F_n,&\text{\rm if $n\equiv 0$ or $3\pmod 5$};  \\ 
 ( - 1)^{\left\lfloor {(n + 1)/5} \right\rfloor + 1} F_{n - 1},&\text{\rm if $n\equiv 1$ or $2\pmod 5$};\\
 0, &\text{\rm if $n\equiv 4\pmod 5$}. 
 \end{cases}
\end{align*}
\end{corollary}
\begin{corollary}
If $n$ is a positive integer, then
\begin{align*}
&\sum_{k = 0}^{\left\lfloor {n/2} \right\rfloor } {( - 1)^k \binom{n - k}kF_{n - 2k + 1} }  =  \begin{cases}
 0,&\text{\rm if $n\equiv 4\pmod 5$}; \\ 
 ( - 1)^{\left\lfloor {(n + 1)/5} \right\rfloor },&\text{\rm otherwise}; 
 \end{cases}\\
&\sum_{k = 0}^{\left\lfloor {n/2} \right\rfloor } {( - 1)^k \binom{n - k}kF_{n - 2k - \delta} }  = 0,
\end{align*}
where
\begin{equation*}
\delta=\begin{cases}
 0,&\text{\rm if $n\equiv 0$ or $3\pmod 5$};  \\ 
 1,&\text{\rm if $n\equiv 1$ or $2\pmod 5$}. 
 \end{cases}
\end{equation*}
\end{corollary}

\section{Fibonacci sums modulo 5 from Waring formulas}

This section is based on utilizing the following trigonometric identities with the use of Waring formulas.
\begin{lemma}\label{lem.pci5uw2}
If $n$ is a positive integer, then
\begin{align}
&\sum_{k = 0}^{\left\lfloor {n/2} \right\rfloor } {(- 1)^k\frac{n}{{n - k}}\binom{n - k}{k} 2^{n - 2k - 1} \cos^{n - 2k} x} = \cos  nx,  \label{eq.q5yiyk7}\\
&\sum_{k = 0}^{(n - 1)/2} {(- 1)^{(n - 1)/2- k} \frac{ n}{{n - k}}\binom{n - k}{k} 2^{n - 2k - 1} \sin^{n - 2k} x} = \sin  nx,\quad\text{\rm $n$ odd},\label{eq.q56i40r}\\
&\sum_{k = 0}^{n/2} {(- 1)^{n/2-k} \frac{n}{{n - k}}\binom{n - k}{k} 2^{n - 2k - 1} \sin^{n - 2k} x} = \cos  nx,\quad\text{\rm  $n$ even}\label{eq,nvk6rjz}.
\end{align}
\end{lemma}
\begin{proof}
Consider the Waring formula
\begin{equation*}
\sum_{k = 0}^{\left\lfloor {n/2} \right\rfloor } {( - 1)^k\frac{{ n}}{{n - k}}\binom{n - k}k(x_1  + x_2 )^{n - 2k} (x_1 x_2 )^k }  = x_1^n  + x_2^n.
\end{equation*}
Let $i$ be the imaginary unit. The choice
$x_1  = e^{ix}/2$, $x_2  = e^{ - ix}/2$ produces \eqref{eq.q5yiyk7}, while the choice
$x_1  = e^{ix}/(2i) $, $x_2  = -e^{ - ix}/{(2i)}$ gives
$x_1  + x_2  = \sin x$, $x_1 x_2  = 1/4$, and
\begin{equation*}
x_1^n  + x_2^n  =  \begin{cases}
 ( - 1)^{(n - 1)/2} 2^{1 - n} \sin nx ,&\text{if $n$ is odd};\\ 
 ( - 1)^{n/2} 2^{1 - n} \cos nx,&\text{if $n$ is even}; 
 \end{cases} 
\end{equation*}
and hence \eqref{eq.q56i40r} and \eqref{eq,nvk6rjz}.
\end{proof}
\begin{lemma}
If $n$ is a positive integer, then
\begin{align}
&\sum_{k = 0}^{\left\lfloor {n/2} \right\rfloor } {(- 1)^k \binom{n - k}{k} 2^{n - 2k} \cos^{n - 2k} x} = \frac{{\sin ((n + 1)x)}}{{\sin x}},
\label{eq.yyyaaa}\\
&\sum_{k = 0}^{(n - 1)/2} {(- 1)^{(n - 1)/2-k} \binom{n - k}{k} 2^{n - 2k} \sin^{n - 2k} x} = \frac{{\sin ((n + 1)x)}}{{\cos x}},\quad\text{\rm  $n$ odd},\nonumber\\
&\sum_{k = 0}^{n/2} {(- 1)^{n/2-k} \binom{n - k}{k} 2^{n - 2k} \sin^{n - 2k} x} = \frac{{\cos((n + 1)x)}}{{\cos x}},\quad\text{\rm $n$ even}.\nonumber
\end{align}
\end{lemma}
\begin{proof}
Similar to the proof of Lemma \ref{lem.pci5uw2}. We use the dual to the Waring formula
\begin{equation*}
\sum_{k = 0}^{\left\lfloor {n/2} \right\rfloor } {( - 1)^k \binom {n - k}k(x_1  + x_2 )^{n - 2k} (x_1 x_2 )^k } 
= \frac{{x_1^{n + 1} - x_2^{n + 1} }}{{x_1 - x_2 }}.
\end{equation*}
\end{proof}
\begin{theorem} \label{th7}
If $n$ is a positive integer and $t$ is any integer, then
\begin{align*}
&\sum_{k = 0}^{\left\lfloor {n/2} \right\rfloor } (-1)^{n-k} \frac{n}{n - k} \binom{n - k}{k} F_{n-2k+t} =  
\begin{cases}
  2  F_t, & \text{\rm  if $n \equiv 0$}\pmod 5;\\ 
 -F_{t+1}, & \text{\rm if $n \equiv 1$ or $4$}\pmod 5; \\ 
 F_{t-1}, & \text{\rm if $n \equiv 2$ or $3$}\pmod 5; 
\end{cases} 
\\
&\sum_{k = 0}^{\left\lfloor {n/2} \right\rfloor } (-1)^{n-k} \frac{n}{n - k} \binom{n - k}{k} L_{n-2k+t} =  
\begin{cases}
  2 L_t, & \text{\rm if $n \equiv 0$}\pmod 5;\\ 
 -L_{t+1}, & \text{\rm if $n \equiv 1$ or $4$}\pmod 5; \\ 
 L_{t-1}, & \text{\rm if $n \equiv 2$ or $3$}\pmod 5. 
\end{cases} 
\end{align*}
\end{theorem}
\begin{proof}
We apply equation \eqref{eq.q5yiyk7}. Inserting $x=\pi/5$ and $x=3\pi/5$, respectively,
and keeping in mind the trigonometric identity
$\cos 3x = 4\cos^3 x - 3\cos x$
we end with
\begin{equation*}
\sum_{k = 0}^{\left\lfloor {n/2} \right\rfloor } (-1)^{n-k} \frac{n}{n - k} \binom{n - k}{k} \alpha^{n-2k+t} =  
\begin{cases}
  2\alpha^t, & \text{if $n \equiv 0$}\pmod 5;\\ 
 -\alpha^{t+1}, & \text{\rm if $n \equiv 1$ or $4$}\pmod 5; \\ 
 \alpha^{t-1}, & \text{\rm if $n \equiv 2$ or $3$}\pmod 5; 
\end{cases} 
\end{equation*}
and
\begin{equation*}
\sum_{k = 0}^{\left\lfloor {n/2} \right\rfloor } (-1)^{n-k} \frac{n}{n - k} \binom{n - k}{k} \beta^{n-2k+t} =  
\begin{cases}
  2 \beta^t, & \text{\rm if $n \equiv 0$}\pmod 5;\\ 
 -\beta^t (\alpha^3-3\alpha), & \text{\rm if $n \equiv 1$ or $4$}\pmod 5; \\ 
 -\beta^t (\beta^3-3\beta), & \text{\rm if $n \equiv 2$ or $3$}\pmod 5. 
\end{cases} 
\end{equation*}
To complete the proof simplify the terms in brackets and combine according the Binet form.
\end{proof}

From Theorem \ref{th7} we can immediately obtain the following finite binomial sums.
\begin{corollary}
If $n$ is a positive integer, then
\begin{align*}
&\sum_{k = 0}^{\left\lfloor {n/2} \right\rfloor } (-1)^{n-k} \frac{n}{n - k} \binom{n - k}{k} F_{n-2k} =  
\begin{cases}
  0, & \text{\rm if $n \equiv 0$}\pmod 5;\\ 
 -1, & \text{\rm if $n \equiv 1$ or $4$}\pmod 5; \\ 
 1, & \text{\rm if $n \equiv 2$ or $3$}\pmod 5; 
\end{cases} \\
&\sum_{k = 0}^{\left\lfloor {n/2} \right\rfloor } (-1)^{n-k} \frac{n}{n - k} \binom{n - k}{k} L_{n-2k} =  
\begin{cases}
  4, & \text{\rm if $n \equiv 0$}\pmod 5;\\ 
 -1, & \text{\rm otherwise}; 
\end{cases} 
\end{align*}
and
\begin{align*}
&\sum_{k = 0}^{\left\lfloor {n/2} \right\rfloor } (-1)^{n-k} \frac{n}{n - k} \binom{n - k}{k} F_{n+1-2k} =  
\begin{cases}
  2, & \text{\rm if $n \equiv 0$}\pmod 5;\\ 
 -1, & \text{\rm if $n \equiv 1$ or $4$}\pmod 5; \\ 
 0, & \text{\rm if $n \equiv 2$ or $3$}\pmod 5; 
\end{cases}\\ 
&\sum_{k = 0}^{\left\lfloor {n/2} \right\rfloor } (-1)^{n-k} \frac{n}{n - k} \binom{n - k}{k} L_{n+1-2k} =  
\begin{cases}
 2, & \text{\rm if $n \equiv 0, 2$ or $3$}\pmod 5;\\ 
 -3, & \text{\rm  otherwise}. 
\end{cases} 
\end{align*}
\end{corollary}
\begin{remark}
Identities \eqref{Waring1} and \eqref{Waring2} in Theorem \ref{thmxxxyyy1} can also be obtained straightforwardly 
by evaluating the trigonometric identity \eqref{eq.yyyaaa} at $x=\pi/5$ and $x=3\pi/5$, respectively, 
while using \eqref{eq.efvrm4x} and \eqref{eq.efvrm4y}. 
\end{remark}

\section{Fibonacci sums modulo 5 from Chebyshev polynomials}

For any integer $n\geq0$, the Chebyshev polynomials $\{T_n(x)\}_{n\geq0}$ of the first kind are defined 
by the second-order recurrence relation \cite{Mason}
\begin{equation*}\label{T-def}
T_{n+1}(x) =2x T_n(x) - T_{n-1}(x),\quad n\geq2,\quad T_0(x)=1,\,\,\, T_1(x)=x, 
\end{equation*}
while the Chebyshev polynomials $\{U_n(x)\}_{n\geq0}$ of the second kind are defined by
\begin{equation*}\label{U-def}
U_{n+1}(x) = 2xU_n(x) - U_{n-1}(x),\quad n\geq2, \quad U_0(x)=1,\,\,\, U_1(x)=2x.
\end{equation*}
The Chebyshev polynomials possess the representations    
\begin{align*}
&T_n(x) = \sum_{k=0}^{\lfloor{n}/{2}\rfloor}{n\choose 2k}(x^2-1)^kx^{n-2k},\\
&U_n(x) = \sum_{k=0}^{\lfloor{n}/{2}\rfloor}{n+1\choose 2k+1}(x^2-1)^kx^{n-2k},
\end{align*}
and have the exact Binet-like formulas
\begin{align*}
&T_n (x) = \frac{1}{2} \left( (x + \sqrt {x^2 - 1} )^n + (x - \sqrt {x^2 - 1} )^n \right),  
\\
&U_n (x) = \frac{1}{2\sqrt {x^2  - 1}} \left( (x + \sqrt {x^2 - 1} )^{n + 1} - (x - \sqrt {x^2 - 1} )^{n + 1} \right).
\end{align*}

The properties of Chebyshev polynomials of the first and second kinds have been studied extensively in the literature. 
The reader can find in the recent papers \cite{Fan,Frontczak,Kilic2,Li-Wenpeng,Li,Zhang} additional information about them, especially about their products, convolutions, power sums as well as their
connections to Fibonacci numbers and polynomials. 
\begin{lemma}
For all $x\in\mathbb{C}$ and a positive integer $n$, we have the following identities:
\begin{align}
&n\sum_{k=0}^n (-1)^k \frac {4^k}{n + k}\binom{n + k}{n - k} \sin^{2k}\!\Big(\frac{x}{2}\Big) = \cos nx,\label{eq.n1t72km}\\
&n\sum_{k=0}^n (-1)^{n-k} \frac {4^k}{n + k}\binom{n + k}{n - k} \cos^{2k}\!\Big(\frac{x}{2}\Big) = \cos nx \label{eq.h0uilmu}.
\end{align}
\end{lemma}
\begin{proof}
Identities \eqref{eq.n1t72km} and \eqref{eq.h0uilmu} are consequences of the identity
\begin{equation}\label{eq.krxoelz}
n\sum_{k=0}^n \frac {(-2)^k}{n + k}\binom{n + k}{n - k} (1\mp x)^k = (\pm 1)^n T_n(x)
\end{equation}
derived in \cite{Adegoke4}.
\end{proof}
\begin{lemma}\label{lem.hewji6w}
If $n$ is a non-negative integer, then
\begin{align*}
T_n \Big(\!- \frac{\alpha}{2} \Big) =  \begin{cases}
  1, & \text{\rm if $n \equiv 0$}\pmod 5;\\
  - \alpha /2, & \text{\rm if $n \equiv 1$ or $4$}\pmod 5;\\ 
  - \beta /2, & \text{\rm if $n \equiv 2$ or $3$}\pmod 5;  
   \end{cases} \\
T_n \Big(\!-\frac{\beta}{2} \Big) =  \begin{cases}
  1, & \text{\rm if $n \equiv 0$}\pmod 5;\\
  - \beta /2, & \text{\rm if $n \equiv 1$ or $4$}\pmod 5;\\ 
  - \alpha /2, & \text{\rm if $n \equiv 2$ or $3$}\pmod 5.
 \end{cases} 
\end{align*}
\end{lemma}
\begin{proof}
Evaluate the identity
$T_n(\cos x)=\cos nx$
at $x=4\pi/5$ and $x=2\pi/5$, in turn.
\end{proof}
\begin{theorem}
If $n$ is a positive integer and $t$ is any integer, then
\begin{equation}\label{eq.dfjg0d9}
\begin{split}
\sum_{k = 1}^{\left\lceil {n/2} \right\rceil } \frac{n}{{n + 2k - 1}}&\binom{n + 2k - 1}{n - 2k + 1}5^k F_{2k + t - 1}   - \sum_{k = 0}^{\left\lfloor {n/2} \right\rfloor } {\frac{n}{{n + 2k}}\binom{n + 2k}{n - 2k}5^k L_{2k + t} }\\
&=  \begin{cases}
  -L_t, & \text{\rm if $n \equiv 0$}\pmod 5;\\ 
  L_{t + 1} /2, & \text{\rm  if $n \equiv 1$ or $4$}\pmod 5;\\ 
  - L_{t - 1} /2, & \text{\rm  if $n \equiv 2$ or $3$}\pmod 5; 
  \end{cases} 
\end{split}
\end{equation}
\begin{equation}\label{eq.lv7l7rr}
\begin{split}
\sum_{k = 1}^{\left\lceil {n/2} \right\rceil } \frac{n}{n + 2k - 1}&\binom{n + 2k - 1}{n - 2k + 1}5^{k-1} L_{2k + t - 1}  - \sum_{k = 0}^{\left\lfloor {n/2} \right\rfloor } {\frac{n}{{n + 2k}}\binom{n + 2k}{n - 2k}5^{k} F_{2k + t} }\\
&=  \begin{cases}
  -F_t, & \text{\rm if $n \equiv 0$}\pmod 5;\\
  F_{t + 1} /2, & \text{\rm if $n \equiv 1$ or $4$}\pmod 5;\\ 
  - F_{t - 1} /2, & \text{\rm  if $n \equiv 2$ or $3$}\pmod 5.
 \end{cases} 
\end{split}
\end{equation}
\end{theorem}
\begin{proof}
Using $x=-\alpha/2$ and $x=-\beta/2$, in turn, in \eqref{eq.krxoelz} with the upper sign gives, in view of Lemma \ref{lem.hewji6w},
\begin{align*}
\sum_{k = 0}^n \frac{n}{{n + k}}\binom{n + k}{n - k}&(\sqrt 5 )^k \big(( - 1)^{k + 1} \lambda \alpha ^{k + t}  - \beta ^{k + t} \big)\\&=  \begin{cases}
  -(\lambda\alpha^t +\beta^t), & \text{\rm if $n \equiv 0$}\pmod 5;\\ 
    (\lambda\alpha^{t + 1} +\beta^{t + 1}) /2, & \text{\rm if $n \equiv 1$ or $4$}\pmod 5;\\ 
  - (\lambda\alpha^{t - 1} +\beta^{t - 1}) /2, & \text{\rm  if $n \equiv 2$ or $3$}\pmod 5; 
 \end{cases} 
\end{align*}
from which \eqref{eq.dfjg0d9} and \eqref{eq.lv7l7rr} now follow upon setting $\lambda=1$ and $\lambda=-1$, in turn, and using the Binet formulas and the summation identity
$\sum\limits_{j = 0}^n {f_j }  = \sum\limits_{j = 0}^{\left\lfloor {n/2} \right\rfloor } {f_{2j} }  + \sum\limits_{j = 1}^{\left\lceil {n/2} \right\rceil } {f_{2j - 1}}.$
\end{proof}

We observe the following special cases of the prior result.
\begin{corollary}
If $n$ is a positive integer, then
\begin{equation*}
\sum_{k = 1}^{\left\lceil {n/2} \right\rceil } {\frac{n}{{n + 2k - 1}}\binom{n + 2k -
		 1}{n - 2k + 1}5^k L_{2k + \delta - 1} }  = \sum_{k = 0}^{\left\lfloor {n/2} \right\rfloor } {\frac{n}{{n + 2k}}\binom{n + 2k}{n - 2k}5^{k + 1} F_{2k + \delta}},
\end{equation*}
where 
\begin{equation*}
\delta=
\begin{cases}
  0, & \text{\rm if $n \equiv 0$}\pmod 5;\\
  -1, & \text{\rm if $n \equiv 1$ or $4$}\pmod 5;\\ 
  1, & \text{\rm if $n \equiv 2$ or $3$}\pmod 5.
 \end{cases} 
\end{equation*}
\end{corollary}
\begin{theorem}\label{th11}
If $n$ is a positive integer and $t$ is any integer, then
\begin{align*}
&\sum_{k = 0}^n {( - 1)^{n-k} \frac{n}{{n + k}}\binom{n + k}{n - k}}L_{2k + t}  = \begin{cases}
 L_t,  &\text{\rm if $n=0\pmod 5$};\\ 
 L_{t - 1} /2,&\text{\rm if $n=1$ or $4\pmod 5$}; \\ 
 -L_{t + 1} /2,&\text{\rm if $n=2$ or $3\pmod 5$}; \\ 
 \end{cases}\\ 
&\sum_{k = 0}^n {( - 1)^{n-k} \frac{n}{{n + k}}\binom{n + k}{n - k}}F_{2k + t}  =  \begin{cases}
 F_t,  &\text{\rm if $n=0\pmod 5$};\\ 
 F_{t - 1} /2,&\text{\rm if $n=1$ or $4\pmod 5$}; \\ 
 -F_{t + 1} /2,&\text{\rm if $n=2$ or $3\pmod 5$}. 
 \end{cases}
\end{align*}
\end{theorem}
\begin{proof}
Set $x=\pi/5$ in \eqref{eq.n1t72km} and use \eqref{eq.kxkaziy} and the fact that $\sin(\pi/10)=-\beta/2$ to obtain
\begin{equation*}
\sum_{k = 0}^n {( - 1)^{n-k} \frac{n}{{n + k}}\binom{n + k}{n - k}\beta^{2k + t} }  = \begin{cases}
 \beta^t,  &\text{\rm if $n=0\pmod 5$};\\ 
 \beta^{t - 1} /2,&\text{\rm if $n=1$ or $4\pmod 5$}; \\ 
 -\beta^{t + 1} /2,&\text{\rm if $n=2$ or $3\pmod 5$}; \\ 
 \end{cases}\
\end{equation*}
from which the results follow by \eqref{eq.fieo8cp}.
\end{proof}

Using Theorem \ref{th11}, we have the following binomial Fibonacci identities modulo 5.
\begin{corollary}
If $n$ is a positive integer, then
\begin{equation*}
\sum_{k = 0}^n { \frac{( - 1)^k}{{n + k}}\binom{n + k}{n - k}}F_{2k + \delta}  =0, 
\end{equation*}
where
\begin{equation*}
\delta= \begin{cases}
 0,  &\text{\rm if $n=0\pmod 5$};\\ 
 1,&\text{\rm if $n=1$ or $4\pmod 5$}; \\ 
 -1,&\text{\rm if $n=2$ or $3\pmod 5$}. 
 \end{cases}
\end{equation*}
\end{corollary}
\begin{lemma}
If $x$ is a complex variable and $n$ is a positive integer, then
\begin{align}
&\sum_{k = 1}^n {( - 1)^{k - 1} \frac{{4^k k}}{{n + k}}\binom{n + k}{n - k}\sin ^{2k - 2} \Big(\frac{x}{2}\Big)}  = \frac{{2\sin nx}}{\sin x},\label{eq.jphe5qh}\\
&\sum_{k = 1}^n {( - 1)^{n-k} \frac{{4^k k}}{{n + k}}\binom{n + k}{n - k}\cos ^{2k - 2} \Big(\frac{x}{2}\Big)}  = \frac{{2\sin nx}}{\sin x}.\label{eq..uyikmgb}
\end{align}
\end{lemma}
\begin{proof}
Identities \eqref{eq.jphe5qh} and  \eqref{eq..uyikmgb} 
come from the following identities derived in \cite{Adegoke4}:
\begin{align*}
&\sum_{k=1}^n (-1)^{n-k}   \frac {2^k k }{n + k}\binom{n + k}{n - k} (1 \mp x)^{k-1} = (\mp 1)^{n-1}  U_{n-1} (x),\\
&\sum_{k=1}^n (-1)^{n-k}  \frac {4^kk}{n + k}\binom{n + k}{n - k} x^{2k-1} = U_{2n-1} (x).
\end{align*}
\end{proof}
\begin{theorem}\label{xyz123}
If $n$ is a positive integer and $n$ is any integer, then
\begin{align*}
&\sum_{k = 1}^n {( - 1)^{k - 1} \frac{k}{{n + k}}\binom{n + k}{n - k}L_{2k + t} }  =  \begin{cases}
 0, &\text{\rm if $n\equiv 0\pmod 5$}; \\ 
 ( - 1)^{\left\lfloor {n/5} \right\rfloor } L_{t + 2}/2, &\text{\rm if $n\equiv 1$ or $4\pmod 5$} ; \\ 
 ( - 1)^{\left\lfloor {n/5} \right\rfloor  + 1} L_{t + 1}/2, &\text{\rm if $n\equiv 2$ or $3\pmod 5$};  
 \end{cases}\\
 &\sum_{k = 1}^n {( - 1)^{k - 1} \frac{k}{{n + k}}\binom{n + k}{n - k}F_{2k + t} }  =  \begin{cases}
 0, &\text{\rm if $n\equiv 0\pmod 5$}; \\ 
 ( - 1)^{\left\lfloor {n/5} \right\rfloor } F_{t + 2}/2, &\text{\rm  if $n\equiv 1$ or $4\pmod 5$} ; \\ 
 ( - 1)^{\left\lfloor {n/5} \right\rfloor  + 1} F_{t + 1}/2, &\text{\rm if $n\equiv 2$ or $3\pmod 5$}.  
 \end{cases} 
\end{align*}
\end{theorem}
\begin{proof}
Set $x=\pi/5$ and $x=3\pi/5$, respectively, in \eqref{eq.jphe5qh}, and use \eqref{eq.efvrm4x} and \eqref{eq.efvrm4y}.
\end{proof}
\begin{remark}
Theorem \ref{xyz123} can also be proved using \eqref{eq..uyikmgb}.
Using the trigonometric identities
$
\sin 2x = 2 \sin x\cos x$ and $\cos 3x =4\cos^3x-3\cos x$
and working with $x=2\pi/5$ and $x=6\pi/5$, respectively, we end with
\begin{equation*}
\begin{split}
& 2\sum_{k = 1}^n ( - 1)^{k - 1} \frac{k}{{n + k}} \binom{n + k}{n - k} L_{2k - 1 + t} \\
& \qquad\qquad
 = \begin{cases}
 0, &\text{\rm if $n\equiv 0\pmod 5$}; \\ 
 ( - 1)^{\left\lfloor {n/5} \right\rfloor } (\alpha^{t+1}-\beta^{t-3}+3\beta^{t-1}), &\text{\rm if $n\equiv 1$ or $4\pmod 5$} ; \\ 
 ( - 1)^{\left\lfloor {n/5} \right\rfloor } (-\alpha^{t}+\beta^{t+4}-3\beta^{t+2}), &\text{\rm if $n\equiv 2$ or $3\pmod 5$}; 
 \end{cases}
\end{split}
\end{equation*}
and
\begin{equation*}
\begin{split}
& 2\sqrt5\sum_{k = 1}^n {( - 1)^{k - 1} \frac{k}{{n + k}}\binom{n + k}{n - k} F_{2k - 1 + t} } \\
& \qquad\qquad = \begin{cases}
 0, &\text{\rm if $n\equiv 0\pmod 5$}; \\ 
 ( - 1)^{\left\lfloor {n/5} \right\rfloor } \big(\alpha^{t+1}+\beta^{t-3}-3\beta^{t-1}\big), &\text{\rm if $n\equiv 1$ or $4\pmod 5$} ; \\ 
 ( - 1)^{\left\lfloor {n/5} \right\rfloor } \big(-\alpha^{t}-\beta^{t+4}+3\beta^{t+2}\big), &\text{\rm if $n\equiv 2$ or $3\pmod 5$}.  
 \end{cases} 
\end{split}
\end{equation*}
To get Theorem \ref{xyz123} simplify the terms in brackets and replace $t$ by $t+1$.
\end{remark}

Applying Theorem \ref{xyz123} yields the following two corollaries.
\begin{corollary}
If $n$ is a positive integer, then
\begin{equation*}
 \sum_{k = 1}^n {( - 1)^{k} \frac{k}{{n + k}}\binom{n + k}{n - k}F_{2k - \delta} } =0,
\end{equation*}
where
\begin{equation*}
\delta = \begin{cases}
 2, &\text{\rm if $n\equiv 1$ or $4\pmod 5$} ; \\ 
 1, &\text{\rm if $n\equiv 2$ or $3\pmod 5$}.  
 \end{cases} 
\end{equation*}
\end{corollary}
\begin{corollary}
If $n$ is a positive integer and $t$ is any integer, then we have: 

If $n\equiv 0\pmod 5$, then
\begin{align*}
&\sum_{k = 0}^{\left\lfloor {(n - 1)/2} \right\rfloor } (- 1)^k \binom{n - k - 1}{k} L_{n - 2k + t} 
= 2 \sum_{k = 1}^n (- 1)^{k - 1} \frac{k}{{n + k}} \binom{n + k}{n - k} L_{2k + t},\\
&\sum_{k = 0}^{\left\lfloor {(n - 1)/2} \right\rfloor } (- 1)^k \binom{n - k - 1}{k} F_{n - 2k + t} 
= 2 \sum_{k = 1}^n (- 1)^{k - 1} \frac{k}{{n + k}} \binom{n + k}{n - k} F_{2k + t};
\end{align*}

if $n\equiv 1 \text{ or } 4 \pmod 5$, then 
\begin{align*}
&\sum_{k = 0}^{\left\lfloor {(n - 1)/2} \right\rfloor } (- 1)^k \binom{n - k - 1}{k} L_{n - 2k + t} 
= 2 \sum_{k = 1}^n (- 1)^{k + 1} \frac{k}{{n + k}} \binom{n + k}{n - k} L_{2k - 1 + t},\\
&\sum_{k = 0}^{\left\lfloor {(n - 1)/2} \right\rfloor } (- 1)^k \binom{n - k - 1}{k} F_{n - 2k + t} 
= 2 \sum_{k = 1}^n (- 1)^{k + 1} \frac{k}{{n + k}} \binom{n + k}{n - k} F_{2k - 1 + t};
\end{align*}

if $n\equiv 2 \text{ or } 3 \pmod 5$, then
\begin{align*}
&\sum_{k = 0}^{\left\lfloor {(n - 1)/2} \right\rfloor } (- 1)^k \binom{n - k - 1}{k} L_{n - 2k + t} 
= 2 \sum_{k = 1}^n (- 1)^{k} \frac{k}{{n + k}} \binom{n + k}{n - k} L_{2k + 1 + t},\\
&\sum_{k = 0}^{\left\lfloor {(n - 1)/2} \right\rfloor } (- 1)^k \binom{n - k - 1}{k} F_{n - 2k + t} 
= 2 \sum_{k = 1}^n (- 1)^{k} \frac{k}{{n + k}} \binom{n + k}{n - k} F_{2k + 1 + t}.
\end{align*}
\end{corollary}
\begin{proof}
Compare Theorem \ref{xyz123} with Theorem \ref{thmxxxyyy1}.
\end{proof}
\begin{lemma}\label{lem9}
If $n$ is a non-negative integer, then
\begin{equation}\label{eq.dom6fjv}
\sum_{k = 0}^n {( - 1)^{n-k} 4^k \binom{n + k}{n - k}\cos ^{2k} x}  =  \frac{{\sin ((2n + 1)x)}}{{\sin x}}.
\end{equation}
\end{lemma}
\begin{proof}
Evaluate the identity \cite{Adegoke4}
\begin{equation*}
\sum_{k = 0}^n {( - 1)^{n-k} 4^k \binom{n + k}{n - k}x^{2k} }  =  U_{2n}(x)
\end{equation*}
at $x=\cos x$.
\end{proof}
\begin{lemma}\label{lem.qdc1y7v}
If $n$ is an integer, then
\begin{equation*}
\frac{{\sin \big((2n + 1)\pi /5\big)}}{{\sin (\pi /5)}} =  
\begin{cases}
 1, &\text{\rm if $n\equiv 0\pmod 5$};\\ 
 \alpha , &\text{\rm if $n\equiv 1\pmod 5$};\\ 
  0, &\text{\rm if $n\equiv 2\pmod 5$};\\ 
  - \alpha , &\text{\rm if $n\equiv 3\pmod 5$};\\ 
 - 1, &\text{\rm if $n\equiv 4\pmod 5$}. 
 \end{cases}
\end{equation*}
\end{lemma}

From Lemmas \ref{lem9} and \ref{lem.qdc1y7v} we can deduce the following Fibonacci and Lucas binomial  identities modulo 5.
\begin{theorem}
If $n$ is a non-negative integer and $t$ is any integer, then
\begin{equation*}
\sum_{k = 0}^n {( - 1)^{n-k} \binom{n + k}{n - k}L_{2k + t} } =  
\begin{cases}
 L_t, &\text{\rm if $n\equiv 0\pmod 5$};\\
 L_{t + 1} , &\text{\rm if $n\equiv 1\pmod 5$};\\ 
 0, &\text{\rm if $n\equiv 2\pmod 5$};\\ 
 - L_{t + 1} , &\text{\rm if $n\equiv 3\pmod 5$};\\
 -L_t , &\text{\rm if $n\equiv 4\pmod 5$}; 
 \end{cases}
\end{equation*}
\begin{equation*}
\sum_{k = 0}^n {( - 1)^{n-k} \binom{n + k}{n - k}F_{2k + t} } =  
\begin{cases}
F_t, &\text{\rm if $n\equiv 0\pmod 5$};\\
 F_{t + 1} , &\text{\rm if $n\equiv 1\pmod 5$};\\ 
 0, &\text{\rm if $n\equiv 2\pmod 5$};\\ 
  - F_{t + 1} , &\text{\rm if $n\equiv 3\pmod 5$};\\
  -F_t , &\text{\rm if $n\equiv 4\pmod 5$}.  
 \end{cases}
\end{equation*}
\end{theorem}
\begin{proof}
Set $x=\pi/5$ in \eqref{eq.dom6fjv} and use Lemma \ref{lem.qdc1y7v}.
\end{proof}
\begin{lemma}[{\cite[(41.2.16.1)]{hansen}}]
If $n$ is a positive integer and $x$ is any variable, then
\begin{equation}\label{eq.cscwk7d}
\sum_{k = 1}^n {\frac{{( - 1)^k }}{{\cos x - \cos (\pi k/n)}}}  = \frac{1}{2}\left( {\frac{1}{{1 - \cos x}} + \frac{{( - 1)^n }}{{1 + \cos x}}} \right) - \frac{n}{\sin x\sin nx}.
\end{equation}
\end{lemma}

Further interesting identities involving Fibonacci and Lucas numbers are stated in the next
theorem.
\begin{theorem}\label{thm.z1mhbc2}
If $n$ is a positive integer and $t$ is any integer, then
\begin{equation*}
\begin{split}
&\sum_{k = 1}^n {\frac{( - 1)^{k - 1} \big( {L_{t - 1}  + 2L_t \cos (\pi k/n)} \big)}{{4\cos ^2 (\pi k/n) - 2\cos (\pi k/n) - 1}}} \\
&\qquad\qquad= \frac{1}{2}\big( {L_{t + 2}  + ( - 1)^n F_{t - 1} } \big) - 2( - 1)^{\left\lfloor {n/5} \right\rfloor}n\cdot\begin{cases}
 0,&\text{\rm if $n\equiv 0\pmod5$}; \\ 
 F_{t + 1},&\text{\rm  if $n\equiv 1$ or $4\pmod 5$};  \\ 
 F_t,&\text{\rm  if $n\equiv 2$ or $3\pmod 5$};
 \end{cases} 
\end{split}
\end{equation*}
\begin{equation*}
\begin{split}
&\sum_{k = 1}^n {\frac{{( - 1)^{k - 1} \big( {F_{t - 1}  + 2F_t \cos (\pi k/n)} \big)}}{{4\cos ^2 (\pi k/n) - 2\cos (\pi k/n) - 1}}} \\
&\qquad\qquad\,\,= \frac{1}{2}\Big( {F_{t + 2}  + \frac{( - 1)^n}5 L_{t - 1} } \Big) - \frac{2( - 1)^{\left\lfloor {n/5} \right\rfloor }}{5}n\cdot 
\begin{cases}
 0,&\text{\rm if $n\equiv 0\pmod5$}; \\ 
  L_{t + 1},&\text{\rm  if $n\equiv 1$ or $4\pmod 5$};  \\ 
  L_t,&\text{\rm if $n\equiv 2$ or $3\pmod 5$}.
 \end{cases} 
\end{split}
\end{equation*}
\end{theorem}
\begin{proof}
Set $x=\pi/5$ and $x=3\pi/5$, in turn, in \eqref{eq.cscwk7d} to obtain
\begin{equation*}
2\sum_{k = 1}^n {\frac{{( - 1)^k }}{{\alpha  - 2\cos (\pi k/n)}}}  = \frac{1}{{2 - \alpha }} + \frac{{( - 1)^n }}{{2 + \alpha }} - \frac{4n\alpha}{\sqrt5} \frac{{\sin (\pi /5)}}{{\sin (n\pi /5)}}
\end{equation*}
and
\begin{equation*}
2\sum_{k = 1}^n {\frac{{( - 1)^k }}{{\beta  - 2\cos (\pi k/n)}}}  = \frac{1}{{2 - \beta }} + \frac{{( - 1)^n }}{{2 + \beta }} + \frac{4n\beta}{\sqrt5} \frac{{\sin (3\pi /5)}}{{\sin (3n\pi /5)}};
\end{equation*}
from which the identities follow.
\end{proof}

By setting $t=0$ and $t=1$ in Theorem \ref{thm.z1mhbc2}, we obtain the following.
\begin{corollary}
If $n$ is a positive integer, then
\begin{align*}
&\sum_{k = 1}^n {\frac{{( - 1)^{k - 1} \left( {4\cos (\pi k/n) - 1} \right)}}{{4\cos ^2 (\pi k/n) - 2\cos (\pi k/n) - 1}}} \\
&\qquad\qquad = \frac{{3 + ( - 1)^n }}{2} - 2( - 1)^{\left\lfloor {n/5} \right\rfloor }\,n\cdot
\begin{cases}
 0, &\text{\rm if $n\equiv 0$, $2$ or $3\pmod 5$};\\ 
 1,&\text{\rm otherwise};  
\end{cases}\\
&\sum_{k = 1}^n {\frac{{( - 1)^{k - 1} }}{{4\cos ^2 (\pi k/n) - 2\cos (\pi k/n) - 1}}}\\ 
&\qquad\qquad = \frac{{5 - ( - 1)^{n} }}{{10}} - \frac{2( - 1)^{\left\lfloor {n/5} \right\rfloor }}{5}\, n\cdot 
\begin{cases}
 0, &\text{\rm if $n\equiv 0\pmod 5$}; \\ 
 1,&\text{\rm if $n\equiv 1$ or $4\pmod 5$}; \\ 
 2,&\text{\rm if $n\equiv 2$ or $3\pmod 5$};  
\end{cases}\\
&\sum_{k = 1}^n {\frac{{( - 1)^{k - 1}  {\cos^2 (\pi k/2n)} }}{{4\cos ^2 (\pi k/n) - 2\cos (\pi k/n) - 1}}} = \frac12 - \frac {( - 1)^{\left\lfloor {n/5} \right\rfloor }}{2}\,n\cdot
\begin{cases}
 0, &\text{\rm if $n\equiv 0\pmod 5$};\\ 
 1,&\text{\rm otherwise};  
\end{cases}\\
&\sum_{k = 1}^n {\frac{( - 1)^{k - 1} \cos (\pi k/n)}{{4\cos ^2 (\pi k/n) - 2\cos (\pi k/n) - 1}}}\\ 
&\qquad\qquad = \frac{{5 + ( - 1)^n }}{10} - \frac{( - 1)^{\left\lfloor {n/5} \right\rfloor }}{5}\,n\cdot 
\begin{cases}
 0, &\text{\rm if $n\equiv 0\pmod 5$}; \\ 
 3,&\text{\rm if $n\equiv 1$ or $4\pmod 5$}; \\ 
 1,&\text{\rm if $n\equiv 2$ or $3\pmod 5$}. 
\end{cases}
\end{align*}
\end{corollary}

\section{Some additional observations}

We close this paper with some additional observations leading to possibly new 
series representations of the constant $\alpha$ involving Bernoulli polynomials. 
Recall that Bernoulli polynomials $B_n(t)$, $n\geq 0$, may be defined by the 
\begin{equation*}
B_n(t) = \sum_{k=0}^n \binom {n}{k} B_{n-k} t^k,
\end{equation*}
where $B_n$ is the $n$th Bernoulli number, defined by the power series
\begin{equation*}
\frac{z}{e^z - 1} = \sum_{n=0}^\infty B_n \frac{z^n}{n!}, \quad |z|< 2\pi.
\end{equation*}
We have $B_n(1)=B_n(0)=B_n$ for all $n\geq 2$ and $B_{2n+1}=0$ for all $n\geq 1$. 
\begin{theorem}
Let $m$ be a non-negative integer. Then
\begin{align}
&\sum_{k=0}^\infty (-1)^k \frac{2^{2k+1}}{(2k+1)!}  \frac{\pi^{2k}}{25^k} B_{2k+1}\Big (\frac{5m}{2}\Big )  
= (-1)^{m-1},\nonumber\\
&\sum_{k=0}^\infty (-1)^k \frac{2^{2k}}{(2k+1)!}  \frac{\pi^{2k}}{25^k}B_{2k+1}\Big (\frac{5m}{2} + \frac12\Big )  
= 0,\nonumber\\
&\sum_{k=0}^\infty (-1)^k \frac{2^{2k+1}}{(2k+1)!}  \frac{\pi^{2k}}{25^k}B_{2k+1}\Big (\frac{5m}{2} + 1\Big )  
= (-1)^{m},\nonumber\\
\label{series1_ber}
&\sum_{k=0}^\infty (-1)^k \frac{2^{2k+1}}{(2k+1)!}  \frac{\pi^{2k}}{25^k} B_{2k+1}\Big (\frac{5m}{2}+\frac32\Big )  
= (-1)^{m} \alpha,
\end{align}
and
\begin{align}\label{series2_ber}
&\sum_{k=0}^\infty(-1)^k \frac{2^{2k+1}}{(2k+1)!}  \frac{\pi^{2k}}{25^k} B_{2k+1}\Big (\frac{5m}{2}+2\Big  )  
= (-1)^{m} \alpha.
\end{align}
\end{theorem}
\begin{proof}
Combine \eqref{eq.efvrm4x} with the representation \cite[Eq.\ (2.5)]{Leeming}
\begin{equation}\label{sin/sin}
\frac{\sin xt}{\sin t} = \sum_{k=0}^\infty (-1)^k \frac{2^{2k+1}}{(2k+1)!}  B_{2k+1}\Big (\frac{1+x}{2}\Big ) t^{2k}, \quad |t|<\pi.
\end{equation}
\end{proof}

When $m=0$ then from \eqref{series1_ber} and  \eqref{series2_ber} we get the special series:
\begin{align*}
&\sum_{k=0}^\infty (-1)^k \frac{2^{2k+1}}{(2k+1)!}  \frac{\pi^{2k}}{25^k}  B_{2k+1}\Big (\frac{3}{2}\Big ) = \alpha,\\
&\sum_{k=0}^\infty (-1)^k \frac{2^{2k+1}}{(2k+1)!}  \frac{\pi^{2k}}{25^k} B_{2k+1} (2)  = \alpha.
\end{align*}

From Raabe's formula
\begin{equation*}
B_n(a x) = a^{n-1} \sum_{k=0}^{a-1} B_n \Big (x + \frac{k}{a}\Big )
\end{equation*}
we get
\begin{equation*}
B_{2k+1}(2) = 2^{2k}\Big (B_{2k+1}(1) + B_{2k+1}\Big (\frac{3}{2}\Big )\Big )
\end{equation*}
and
\begin{equation*}
\sum_{k=0}^\infty (-1)^k \frac{2^{2k+1}}{(2k+1)!}\frac{\pi^{2k}}{25^k}   B_{2k+1}\Big (\frac{3}{2}\Big ) = \alpha,
\end{equation*}
\begin{equation*}
\sum_{k=1}^\infty(-1)^k \frac{2^{4k+1}}{(2k+1)!}  \frac{\pi^{2k}}{25^k}B_{2k+1} \Big (\frac{3}{2}\Big )   = \alpha - 3 = \sqrt{5}\beta.
\end{equation*}
But making use of $B_n(t+1)-B_n(t)=n t^{n-1}$ we see that
\begin{equation*}
B_{2k+1}\Big (\frac{3}{2}\Big ) = \frac{2k+1}{2^{2k}}
\end{equation*}
and thus the series turn into
\begin{equation}\label{last1}
\sum_{k=1}^\infty \frac{(-1)^k}{(2k)!}  \frac{\pi^{2k}}{25^k} = \frac{\alpha}{2} - 1=-\frac{\beta^2}{2}
\end{equation}
and
\begin{equation}\label{last2}
\sum_{k=1}^\infty (-1)^k \frac{2^{2k+1}}{(2k)!} \frac{\pi^{2k}}{25^k} = \sqrt{5}\beta.
\end{equation}

The series \eqref{last1} and \eqref{last2} are essentially $\cosh (i \pi /5)= \cos (\pi/5)=\alpha/2$ and $\cosh(2i \pi/5)=\cos(2\pi/5)=-\beta/2$ 
which we encountered at the beginning of the paper.

Combining \eqref{eq.efvrm4y} with \eqref{sin/sin} we have the following theorem. The details of we
leave to the reader.
\begin{theorem}
	Let $m$ be a non-negative integer. Then
	\begin{align}
	&\sum_{k=0}^\infty (-1)^k \frac{2^{2k+1}}{(2k+1)!}  \frac{9^k\pi^{2k}}{25^k}B_{2k+1}\Big (\frac{5m}{2}\Big )  
	= (-1)^{m-1},\nonumber\\
	&\sum_{k=0}^\infty (-1)^k \frac{2^{2k}}{(2k+1)!}  \frac{9^k\pi^{2k}}{25^k}B_{2k+1}\Big (\frac{5m}{2} + \frac{1}{2}\Big )  
	= 0,\nonumber\\
	&\sum_{k=0}^\infty (-1)^k \frac{2^{2k+1}}{(2k+1)!}  \frac{9^k\pi^{2k}}{25^k} B_{2k+1}\Big (\frac{5m}{2}+1\Big )  
	= (-1)^{m},\nonumber\\
	\label{series1_ber2}
	&\sum_{k=0}^\infty (-1)^k \frac{2^{2k+1}}{(2k+1)!}  \frac{9^k\pi^{2k}}{25^k} B_{2k+1}\Big (\frac{5m}{2}+\frac32\Big )  
	= (-1)^{m} \beta,
	\end{align}
	and
	\begin{align}\label{series2_ber2}
	&\sum_{k=0}^\infty(-1)^k \frac{2^{2k+1}}{(2k+1)!}  \frac{\pi^{2k}}{25^k} B_{2k+1}\Big (\frac{5m}{2}+2\Big  )  
	= (-1)^{m} \beta.
	\end{align}
\end{theorem}

Finally, we obtain the following special series as a consequence of  \eqref{series1_ber2} and  \eqref{series2_ber2}:
\begin{align*}
&\sum_{k=0}^\infty (-1)^k \frac{2^{2k+1}}{(2k+1)!}  \frac{9^k\pi^{2k}}{25^k}  B_{2k+1}\Big (\frac{3}{2}\Big ) = \beta,\\
&\sum_{k=0}^\infty (-1)^k \frac{2^{2k+1}}{(2k+1)!}  \frac{9^k\pi^{2k}}{25^k} B_{2k+1} (2)  = \beta.
\end{align*}

\section{Concluding comments}

In this paper, we presented new  closed forms for some types of finite Fibonacci and Lucas sums involving different kinds of binomial coefficients and depending on the modulo 5 nature of the upper summation limit. To prove our results, we applied some trigonometric identities utilizing Waring formulas and Chebyshev polynomials of the first and second kinds. 

Using similar techniques, we can generalize our findings to more common number sequences. Let us give, for example, a generalization of Theorems \ref{th1}, \ref{th7} and \ref{th11} to the case of the gibonacci (generalized Fibonacci) sequence  defined by the recurrence 
$G_n=G_{n-1}+G_{n-2},$ $n\geq2$, with  $G_0=a$ and $G_1=b$, where $a$ and $b$ are arbitrary \cite{Koshy,Vajda}.  Note that $F_n$ corresponds to the case of $G_n$ when $a = 1$ and $b = 0$, while $L_n$ to the case when $a = 1$
and $b = 2$. The following identities modulo 5 hold for positive integer $n$ and any integer $t$: 
\begin{align*}
&n\sum_{k = 1}^{\left\lfloor {n/2} \right\rfloor } \frac{{( - 1)^{k - 1}}}{k}\binom{n - k - 1}{k - 1}G_{n - 2k + t} 
=  \begin{cases}
G_{n + t}  - ( - 1)^{n} 2G_t,&\text{\rm if $n\equiv 0\pmod 5$};  \\ 
G_{n + t}  + ( - 1)^n G_{t + 1},&\text{\rm if $n\equiv 1$ or $4\pmod 5$};  
\\ 
G_{n + t}  - ( - 1)^{n} G_{t - 1},&\text{\rm if $n\equiv 2$ or $3\pmod 5$}; 
\end{cases} \\
&n\sum_{k = 0}^{\left\lfloor {n/2} \right\rfloor }  \frac{(-1)^{n-k}}{n - k} \binom{n - k}{k} G_{n-2k+t} =  
\begin{cases}
2  G_t, & \text{\rm  if $n \equiv 0$}\pmod 5;\\ 
-G_{t+1}, & \text{\rm if $n \equiv 1$ or $4$}\pmod 5; \\ 
G_{t-1}, & \text{\rm if $n \equiv 2$ or $3$}\pmod 5; 
\end{cases}
\end{align*}
and
\begin{align*}
&n\sum_{k = 0}^n {\frac{( - 1)^{n-k} }{n + k}\binom{n + k}{n - k}}G_{2k + t}  = \begin{cases}
G_t,  &\text{\rm if $n=0\pmod 5$};\\ 
G_{t - 1} /2,&\text{\rm if $n=1$ or $4\pmod 5$}; \\ 
-G_{t + 1} /2,&\text{\rm if $n=2$ or $3\pmod 5$}. 
\end{cases} 
\end{align*}

\end{document}